\documentclass[11pt, reqno, a4paper]{amsart}
\usepackage[margin=1.2in]{geometry}
\numberwithin{equation}{section}
\usepackage{amssymb,amsfonts,amsthm}
\usepackage[numbers,sort&compress]{natbib}
\usepackage[utf8]{inputenc}
\usepackage{listings}
\usepackage{bm}
\usepackage[hyperpageref]{backref}
\usepackage{esint}
\usepackage{color}
\usepackage[bookmarks]{hyperref}
\usepackage{siunitx}
\usepackage{bigints}
\usepackage{tikz}
\usepackage{subcaption}
\usepackage{hyperref}
\usepackage{mathtools}

\allowdisplaybreaks[4]

\newtheoremstyle{myremark}{10pt}{10pt}{}{}{\bfseries}{.}{.5em}{}

\newtheorem{theorem}{Theorem}[section]

\newtheorem{lemma}[theorem]{Lemma}

\theoremstyle{definition}

\newtheorem{remark}{Remark}[section]

\usepackage{hyperref}

\allowdisplaybreaks[4]

\begin{document}

\title[Quantitative stability for the critical Hardy inequality]{Improved Quantitative Stability for the Critical Hardy Inequality}

\author{Vivek Sahu}

\address{ Theoretical Statistics and Mathematics Unit,
Indian Statistical Institute, Delhi Centre, S.J. Sansanwal Marg, New Delhi, Delhi 110016, India}
\email{vivek@isid.ac.in,  viiveksahu@gmail.com}

\subjclass[2020]{26D10; 46E35; 35A23}

\keywords{Quantitative stability,  Hardy inequality, Critical case, logarithmic weight}

\date{}

\dedicatory{}

\begin{abstract}
We establish a quantitative stability estimate for the critical Hardy inequality on bounded domains containing the origin. Our result improves the existing quantitative stability estimate by reducing the exponent in the distance function from $N^{2}$ to $N$, replacing the Lorentz--Zygmund framework with the Luxemburg norm of the critical exponential Orlicz space $\operatorname{Exp}L^{\frac{N}{N-1}}(\Omega)$, and avoiding any cut-off modification of the virtual extremizers. As a consequence, we also obtain an improved quantitative stability estimate for the critical Hardy inequality with the logarithmic weight considered by Cianchi and Ferone, where the exponent is likewise reduced from $N^{2}$ to $N$ and the distance is measured directly from the virtual extremizer without truncation. The proof is completely rearrangement-free and relies on a critical Hardy inequality with a remainder term, scale-invariant Sobolev inequalities, and a refined dyadic summation argument. These ingredients yield stronger quantitative stability estimates for both forms of the critical Hardy inequality.
\end{abstract}

\maketitle


\section{Introduction}\label{Section 1}

Hardy inequalities are among the most fundamental functional inequalities in analysis and partial differential equations. One remarkable feature of the sharp Hardy inequality is that its optimal constant is not attained in the natural Sobolev space. This naturally leads to the study of quantitative stability, where one aims to understand the behaviour of functions that are close to virtual extremals and to measure their distance from the family of virtual
extremals.

\smallskip

Let $N\geq 2$ and $1<p<N$. The classical Hardy inequality on
$\mathbb{R}^{N}$ states that
\begin{equation}\label{eq:classical-hardy}
   \int_{\mathbb{R}^{N}}
    |\nabla u(x)|^{p}\,dx \geq  \left(\frac{N-p}{p}\right)^{p}
    \int_{\mathbb{R}^{N}}
    \frac{|u(x)|^{p}}{|x|^{p}}\,dx,
\end{equation}
for every $ u \in \mathcal{D}^{1,p}(\mathbb{R}^{N})$, where
$\mathcal{D}^{1,p}(\mathbb{R}^{N})$ denotes the completion of
$C_{c}^{\infty}(\mathbb{R}^{N})$ with respect to 
$\|\nabla u\|_{L^{p}(\mathbb{R}^{N})}$. The constant $ \left(\frac{N-p}{p}\right)^{p}$
is optimal. However, it is not attained by any nonzero function in the homogeneous Sobolev space $\mathcal{D}^{1,p}(\mathbb{R}^{N})$. The formal extremals associated with \eqref{eq:classical-hardy} are given by
\begin{equation}\label{eq:subcritical-virtual-extremals}
    v_{a}(x)
    =
    a \, |x|^{-\frac{N-p}{p}},
    \qquad
    a\in\mathbb{R} \setminus \{ 0\}.
\end{equation}
These functions do not belong to $\mathcal{D}^{1,p}(\mathbb{R}^{N})$.
Therefore, they are called the \textit{virtual extremals} of the Hardy inequality.

The non-attainment of the optimal constant leads to the problem of
quantitative stability. The main aim is to estimate the Hardy deficit
\begin{equation*}
    \mathcal{H}_{p}(u)
    :=
    \int_{\mathbb{R}^{N}}
    |\nabla u(x)|^{p}\,dx
    -
    \left(\frac{N-p}{p}\right)^{p}
    \int_{\mathbb{R}^{N}}
    \frac{|u(x)|^{p}}{|x|^{p}}\,dx
\end{equation*}
in terms of the distance of $u$ from the family $    \mathcal{M}_{p} := \left\{ a \, |x|^{-\frac{N-p}{p}}: a\in\mathbb{R} \right\}$.

\smallskip

A seminal quantitative improvement of \eqref{eq:classical-hardy} was obtained by Cianchi and Ferone \cite[Theorem $1.1$]{CianchiFerone}. The virtual extremals $v_a$, defined by \eqref{eq:subcritical-virtual-extremals}, belong to the Marcinkiewicz (weak Lorentz) space $L^{p^{*},\infty} (\mathbb{R}^{N})$, where $p^{*}:=\frac{Np}{N-p}$. Motivated by this fact, they introduced the normalized distance
\begin{equation*}
    d_{p}(u)
    :=
    \inf_{a\in\mathbb{R}}
    \frac{
        \|u-v_{a}\|_{L^{p^{*},\infty}(\mathbb{R}^{N})}
    }{
        \|u\|_{L^{p^{*},p}(\mathbb{R}^{N})}
    },
\end{equation*}
where $v_{a}$ is given by \eqref{eq:subcritical-virtual-extremals}, $L^{p^{*},\infty}(\mathbb{R}^{N})$ denotes the Marcinkiewicz (weak Lorentz) space, and $L^{p^{*},p}(\mathbb{R}^{N})$ denotes the Lorentz space. They proved that there exists a
constant $C=C(N,p)>0$ such that
\begin{equation}\label{eq:Cianchi-Ferone-subcritical}
     \mathcal{H}_{p}(u)
    \geq
    C
    \left(
        \int_{\mathbb{R}^{N}}
        \frac{|u(x)|^{p}}{|x|^{p}}\,dx
    \right)
    d_{p}(u)^{2p^{*}},
\end{equation}
for every real-valued weakly differentiable function $u$ on
$\mathbb{R}^{N}$ such that $ |\nabla u|\in L^{p} (\mathbb{R}^{N})$ and $u$ decays to zero at infinity. Thus, if a sequence is close to equality in the Hardy inequality \eqref{eq:classical-hardy}, then it
must be close to the family of virtual extremals in the weak Lorentz space
$L^{p^{*},\infty}(\mathbb{R}^{N})$. The use of this space is natural because
it is the smallest rearrangement-invariant space containing the functions
$v_{a}$ (see \cite[Proposition $2.3$]{CianchiFerone}).

The stability estimate \eqref{eq:Cianchi-Ferone-subcritical} naturally raises
the question of whether the exponent of the distance function can be further
improved. Recently, in \cite{BanerjeeGangulySahu}, the quantitative stability for the fractional Hardy inequality was established in the subcritical case. Moreover, the method developed there was used to revisit the quantitative stability for the classical Hardy inequality, leading to an improvement in the exponent of the distance function. In particular, the following estimate was obtained (see \cite[Theorem $1.3$]{BanerjeeGangulySahu}):
\begin{equation*}
    \mathcal{H}_{p}(u)
    \geq
    C \left(  \int_{\mathbb{R}^{N}} \frac{|u(x)|^{p}}{|x|^{p}}\,dx \right) d_{p}(u)^{\max\{4,2p\}}.
\end{equation*}
This improves the exponent of the distance term from $2p^{*}$ in the result of Cianchi and Ferone \cite[Theorem $1.1$]{CianchiFerone} to $\max\{4,2p\}$, while using the same distance function. 

\smallskip

We now turn to the limiting case $p=N$. In this case, the weight in the Hardy inequality becomes $|x|^{-N}$, which is neither locally integrable near the origin nor integrable at infinity. Therefore, the direct analogue of the Hardy inequality \eqref{eq:classical-hardy} cannot hold on $\mathbb{R}^{N}$. To overcome this difficulty, the critical Hardy inequality is reformulated on bounded domains by incorporating a logarithmic correction that weakens the singularity at the origin. Let $\Omega\subset\mathbb{R}^{N}$ be a bounded domain containing the origin, where $N \geq 2$, and let $\widetilde{R}:=\sup_{x\in\Omega}|x|$. For $R\geq \widetilde{R}$, the critical Hardy inequality takes the form
\begin{equation}\label{Hardy critical Cianchi}
    \int_{\Omega} |\nabla u(x)|^{N}\,dx
    \geq
    \left(\frac{N-1}{N}\right)^{N}
    \bigintsss_{\Omega}
    \frac{|u(x)|^{N}}
    {|x|^{N}\left(1+\log\frac{R}{|x|}\right)^{N}}
    \,dx,
    \qquad \forall \,
    u\in W^{1,N}_{0}(\Omega),
\end{equation}
where $W^{1,N}_{0}(\Omega)$ is the completion of $C^{\infty}_{c}(\Omega)$ with respect to the norm $\| \cdot \|_{W^{1,N}(\Omega)} $. The constant $\left(\frac{N-1}{N}\right)^{N}$ is optimal, but it is not attained. The corresponding family of virtual extremals is given by
\begin{equation}\label{Extremizer cianchi}
    z_{a}(x)
    =
    a\left[
    \left(
        1+\log\frac{R}{|x|}
    \right)^{\frac{N-1}{N}}
    -Q
    \right]_{+},
    \qquad \text{for} \,  x\in\Omega,
\end{equation}
where $a\in\mathbb{R}\setminus\{0\}$ and
$Q>\left(1+\log\frac{R}{r_{\Omega}}\right)^{(N-1)/N)}$, with $r_{\Omega}
:= \sup\left\{ r>0:B_{r}(0)\subset\Omega \right\}$, so that the support of $z_{a}$ is contained in $\Omega$.

Following the same philosophy as in the subcritical case, Cianchi and Ferone
\cite[Theorem $1.2$]{CianchiFerone} measured the distance of a function from the family of
virtual extremals by introducing the distance functional
\begin{equation*}
    d_{C,R,Q}(u)
    :=
    \inf_{a\in\mathbb{R}}
    \bigintsss_{\Omega}
    \left[
    \exp\left(
    \frac{
    C|u(x)-z_{a}(x)|^{\frac{N}{N-1}}
    }
    {
    \|u\|_{L^{\infty,N}(\log L)^{-1}(\Omega),R}^{\frac{N}{N-1}}
    }
    \right)
    -1
    \right]dx,
\end{equation*}
where $C>0$, and $\|\cdot\|_{L^{\infty,N}(\log L)^{-1}(\Omega),R}$ denotes a family of
equivalent norms on the Lorentz--Zygmund space
$L^{\infty,N}(\log L)^{-1}(\Omega)$, defined by
\begin{equation}\label{Lorentz Zygmund Norm}
    \|u\|_{L^{\infty,N}(\log L)^{-1}(\Omega),R}
    := \left(
    \bigintsss_{0}^{|\Omega|}
    \frac{u^{*}(s)^{N}}
    {\left( N+\log\frac{|B_{1}|R^{N}}{s} \right)^{N}}
    \,\frac{ds}{s} \right)^{\frac{1}{N}},
\end{equation}
where $u^{*}$ denotes the decreasing rearrangement of $u$ defined by
\begin{equation}\label{decreasing rearrangement}
    u^{*}(t):=\inf\left\{s\geq 0:\,|\{x\in\Omega:|u(x)|>s\}|\leq t\right\}, \qquad \text{for} \hspace{.2cm} t\geq 0,
\end{equation}
and $|B_{1}|$ denotes the Lebesgue measure of the unit ball in $\mathbb{R}^{N}$.

 It is important to note that the Lorentz--Zygmund norm in
\eqref{Lorentz Zygmund Norm} cannot itself be used to define the distance
from the virtual extremals, since the functions $z_{a} \notin L^{\infty,N}(\log L)^{-1}(\Omega), R$. The exponential expression appearing in the definition of $d_{C,R,Q}$
therefore provides an appropriate way to measure the deviation of $u$ from the family of virtual extremals. They also proved that the Orlicz space associated with the Young function $\Phi(t):=\exp\left(t^{\frac{N}{N-1}}\right)-1$  is the smallest rearrangement-invariant space containing the family of virtual extremals \eqref{Extremizer cianchi}. They then established a quantitative stability estimate for the critical
Hardy inequality. 

They proved that for $R > \widetilde{R}$, there exists a constant
$C=C(N,R,\widetilde{R},Q)>0$ such that
\begin{align}\label{Cianchi Quantitative}
    \int_{\Omega} |\nabla u(x)|^{N}\,dx
    - &
    \left(\frac{N-1}{N}\right)^{N}
    \bigintsss_{\Omega}
    \frac{|u(x)|^{N}}
    {|x|^{N}
    \left(
    1+\log\frac{R}{|x|}
    \right)^{N}}
    \,dx \nonumber \\ &
    \geq
    \left(\frac{N-1}{N}\right)^{N}
    \left(
    \bigintsss_{\Omega}
    \frac{|u(x)|^{N}}
    {|x|^{N}
    \left(
    1+\log\frac{R}{|x|}
    \right)^{N}}
    \,dx
    \right)
    \left(d_{C,R,Q}(u) \right)^{N^{2}}.
\end{align}

Motivated by the above results, in this paper we revisit the quantitative
stability estimate for the critical Hardy inequality established by Cianchi
and Ferone \cite[Theorem~1.2]{CianchiFerone}. Our first main contribution is
to improve the exponent of the distance term from $N^{2}$ to $N$. Moreover,
instead of using the Lorentz--Zygmund norm, we introduce a new remainder term
based on the Luxemburg norm of the Orlicz space
$\operatorname{Exp}L^{\frac{N}{N-1}}(\Omega)$. This provides a stronger and
more natural quantitative stability estimate for the critical Hardy
inequality.

More precisely, we consider the critical Hardy inequality
established by Ioku and Ishiwata \cite[Remark $1.3$]{IokuIshiwata2015}. Let
$\Omega$ be a bounded domain in $\mathbb{R}^{N}$ containing the origin, and
define $\widetilde{R}:=\sup_{x\in\Omega}|x|$. Then, for every $u\in W_{0}^{1,N}(\Omega)$,
\begin{equation}\label{Critical Hardy}
    \int_{\Omega} |\nabla u(x)|^{N}\,dx
    \geq
    \left(\frac{N-1}{N}\right)^{N}
    \bigintsss_{\Omega}
    \frac{|u(x)|^{N}}
    {|x|^{N}\left( \log\frac{\widetilde{R}}{|x|} \right)^{N}}
    \,dx.
\end{equation}
Moreover, \cite[Remark~2.1]{IokuIshiwata2015} observed that the function
\begin{equation}\label{Defn: Omega}
    \omega_{\widetilde{R}}(x)
    :=
    \left(
    \log\frac{\widetilde{R}}{|x|}
    \right)^{\frac{N-1}{N}}
\end{equation}
satisfies the Euler--Lagrange equation associated with the above critical
Hardy inequality. Since $\omega_{\widetilde{R}} \notin W_{0}^{1,N}(\Omega)$, the optimal
constant is not attained. Therefore, $\omega_{\widetilde{R}}$ is called the
\emph{virtual extremizer} of the critical Hardy inequality. It is not difficult to see that $\omega_{\widetilde{R}} \in L^{q}(\Omega)$ for all $1 \leq q < \infty$ and $ \omega_{\widetilde{R}} \in \operatorname{Exp}L^{\frac{N}{N-1}}(\Omega)$, where $\operatorname{Exp}L^{\frac{N}{N-1}}(\Omega)$ is defined in
\eqref{Defn: Exp Norm} (see Lemma \ref{Lemma: On virtual Extremizers} for a proof).

Although the critical Hardy inequality considered here has the same optimal
constant as that of Cianchi and Ferone \cite{CianchiFerone}, the underlying Hardy inequality is different. We consider the critical Hardy inequality of
Ioku and Ishiwata \cite{IokuIshiwata2015}, where the logarithmic weight is
given by $|x|^{-N}\left(\log\widetilde{R}/|x|\right)^{-N}$, which is more singular than the weight $|x|^{-N}\left(1+\log R/|x|\right)^{-N}$ appearing in \cite{CianchiFerone}. Indeed, besides the singularity at the
origin, the weight
$|x|^{-N}\left(\log\widetilde{R}/|x|\right)^{-N}$
also becomes singular as $|x|\to\widetilde{R}$, whereas the weight in
\eqref{Hardy critical Cianchi} is singular only at the origin. The logarithmic weight in \cite{CianchiFerone} was chosen because it allows the use of the Hardy--Littlewood inequality and rearrangement techniques. In contrast, such rearrangement arguments are no longer applicable to the stronger critical Hardy weight, owing to its additional singularity near the boundary. Consequently, the
associated Euler--Lagrange equation is different, leading to a different
family of virtual extremizers. This naturally motivates the introduction of a
new distance function and enables us to establish a new quantitative
stability estimate. As a consequence, we establish a new quantitative stability estimate for the critical Hardy inequality. Although our proof is based on the more singular critical Hardy weight, the stability estimate is ultimately obtained in terms of the logarithmic weight $\omega_{R}$, for $R> \widetilde{R}$.

Let $\Omega$ be a bounded domain in $\mathbb{R}^{N}$ containing the origin. Let $\widetilde{R}=\sup_{x\in\Omega}|x|$ and fix $R>\widetilde{R}$. For $C>0$, we define the distance function
\begin{equation} \mathcal{D}_{N,R,C}(u) := \inf_{a \in \mathbb{R}} \bigintsss_{\Omega} \left[ \exp \left( \frac{C |u(x) - a \ \omega_{R}(x)|^{\frac{N}{N-1}}}{\|u\|^{\frac{N}{N-1}}_{\operatorname{Exp} L^{\frac{N}{N-1}}(\Omega)}} \right) - 1 \right] \, dx , \end{equation}
where $\omega_{R}$ is defined in \eqref{Defn: Omega}, and
$\|u\|_{\operatorname{Exp}L^{\frac{N}{N-1}}(\Omega)}$ denotes the Luxemburg
norm defined in \eqref{Defn: Exp Norm}. Our main result is the
following.

\begin{theorem}\label{Theorem: Quantitative stability critical Hardy}
Let $\Omega$ be a bounded open set containing the origin in $\mathbb{R}^{N}$, where $N \geq 2$, and let $\widetilde{R} = \sup_{x \in \Omega} |x|$ and $R> \widetilde{R}$. Then there exists a constant
$C=C(N,R,\widetilde{R})>0$ such that
\begin{equation*}
\int_{\Omega} |\nabla u(x)|^{N}\,dx
-
\left(\frac{N-1}{N}\right)^{N}
\bigintsss_{\Omega}
\frac{|u(x)|^{N}}
{|x|^{N} \left( \log \frac{R}{|x|} \right)^{N}}
\,dx
\geq
\|u\|^{N}_{\operatorname{Exp}L^{\frac{N}{N-1}}(\Omega)} \, \bigl(\mathcal{D}_{N,R,C}(u)\bigr)^{N},
\end{equation*}
for every $u\in W^{1,N}_{0}(\Omega)$.
\end{theorem}

The above theorem establishes a quantitative stability estimate for the
critical Hardy inequality and improves the quantitative stability result of
Cianchi and Ferone \cite[Theorem $1.2$]{CianchiFerone} in several directions.
First, the exponent in the distance function is improved from $N^{2}$ to $N$.
Second, the distance function is expressed in terms of the Luxemburg norm of the
Orlicz space $\operatorname{Exp}L^{\frac{N}{N-1}}(\Omega)$, rather than the norm
of $L^{\infty,N}(\log L)^{-1}(\Omega)$. This is a genuine improvement since the continuous embedding
\begin{equation*}
    L^{\infty,N}(\log L)^{-1}(\Omega)
\subsetneq 
\operatorname{Exp}L^{\frac{N}{N-1}}(\Omega),
\end{equation*}
and the corresponding norm appears in the denominator of the distance
function. Finally, in \cite[Theorem $1.2$]{CianchiFerone}, the virtual extremizers are truncated by subtracting a sufficiently large positive constant and taking the positive part, so that the resulting functions have support contained in $\Omega$. In contrast, Theorem
\ref{Theorem: Quantitative stability critical Hardy} establishes the
quantitative stability estimate directly in terms of the virtual extremizer
$\omega_{R}$, defined in \eqref{Defn: Omega}, without introducing any cut-off
function.

We also obtain the following improved quantitative stability estimate for the critical Hardy inequality with the logarithmic weight considered in \eqref{Cianchi Quantitative} as established in \cite[Theorem~$1.2$]{CianchiFerone}. Compared to \eqref{Cianchi Quantitative}, our result has two improvements: the exponent of the distance term is reduced from $N^{2}$ to $N$, and the distance is measured directly from the virtual extremizer $\omega_{R}$, without introducing any truncation. Moreover, while \eqref{Cianchi Quantitative} is established for $R>\widetilde{R}$, our result holds for every $R\geq\widetilde{R}$.  Let $u \in W^{1,N}_{0}(\Omega)$ and $R \geq \widetilde{R} = \sup_{\Omega} |x|$. Extending $u$ by zero on $B_{\widetilde{R}}(0) \setminus \Omega$, assume $u \in W^{1,N}_{0}(B_{\widetilde{R}}(0))$. Using the decreasing function $\frac{1}{|x|^{N} \left( 1+ \log R/|x| \right)^{N}}$ and the Hardy--Littlewood inequality (see \cite[Inequality $(3.1)$]{CianchiFerone}), we have
\begin{equation}\label{Hardy Littlewood inequality}
\bigintsss_{B_{\widetilde{R}}(0)}
\frac{|u(x)|^{N}}
{|x|^{N} \left( 1+ \log \frac{R}{|x|} \right)^{N}}
\,dx \leq 
\bigintsss_{B_{\widetilde{R}}(0)}
\frac{|u^{\#}(x)|^{N}}
{|x|^{N} \left( 1+ \log \frac{R}{|x|} \right)^{N}}
\,dx=: \| u\|^{N}_{\mathcal{L}^{N, R}(\Omega)},
\end{equation}
where $u^{\#}$ is the Schwarz symmetrization of $u$ defined by
$u^{\#}(x) = u^{*}(|B_{1}| |x|^{N})$, where $u^{*}$ is the decreasing rearrangement of $u$, defined in \eqref{decreasing rearrangement}, and $|B_{1}|$ is the Lebesgue measure of a unit ball in $\mathbb{R}^{N}$. For $C>0$, define
\begin{equation*}
\widetilde{\mathcal{D}}_{N,R,C}(u)
    :=
    \inf_{a\in\mathbb{R}}
    \bigintsss_{\Omega}
    \left[
    \exp\left(
    \frac{
    C|u(x)- a \, \widetilde{\omega}_{R}(x)|^{\frac{N}{N-1}}
    }
    {
    \| u\|^{\frac{N}{N-1}}_{\mathcal{L}^{N, R}(\Omega)}
    }
    \right)
    -1
    \right]dx,
\end{equation*}
where
\begin{equation}\label{Defn: Omege widetilde}
    \widetilde{\omega}_{R}(x)
    :=
    \left( 1+
    \log\frac{R}{|x|}
    \right)^{\frac{N-1}{N}}.
\end{equation}
Then the following theorem holds.

\begin{theorem}\label{Theorem 2}
Let $\Omega$ be a bounded open set containing the origin in $\mathbb{R}^{N}$, where $N \geq 2$, and let $\widetilde{R} = \sup_{x \in \Omega} |x|$ and $R \geq \widetilde{R}$. Then there exists a constant
$C=C(N,R,\widetilde{R})>0$ such that
\begin{align*}
\int_{\Omega} |\nabla u(x)|^{N}\,dx
-
\left(\frac{N-1}{N}\right)^{N} &
\bigintsss_{\Omega}
\frac{|u(x)|^{N}}
{|x|^{N} \left( 1+ \log \frac{R}{|x|} \right)^{N}}
\,dx \\ &
\geq
\left(\frac{N-1}{N}\right)^{N} \left(
\bigintsss_{\Omega}
\frac{|u(x)|^{N}}
{|x|^{N} \left( 1+ \log \frac{R}{|x|} \right)^{N}}
\,dx \right)  \bigl(\widetilde{\mathcal{D}}_{N,R,C}(u)\bigr)^{N},
\end{align*}
for every $u\in W^{1,N}_{0}(\Omega)$.
\end{theorem}

Our approach is completely rearrangement-free and differs substantially from the method of Cianchi and Ferone \cite{CianchiFerone}. Their proof relies heavily on Hardy--Littlewood inequalities and sophisticated rearrangement arguments. In contrast, our proof is based on a scale-invariant Sobolev inequality, a Hardy inequality with a remainder term, and a careful summation argument over a dyadic decomposition of the domain. These techniques allow us to avoid any use of rearrangement theory while yielding a sharper quantitative stability estimate. 

\begin{remark}
The exponent $N$ naturally arises from the dyadic summation argument in the proof and improves the previously known exponent $N^{2}$. Whether the exponent $N$ is optimal remains an interesting open problem.
\end{remark}

Recent years have witnessed significant developments in Hardy inequalities, critical Hardy inequalities, logarithmic Hardy inequalities, and their refinements; see, for instance, \cite{Adi2002, Adimurthi2005, Adimurthi2009, Adi2025, Adi2006, Barbatis2003, Barbatis2004, Brezis1997-2, Brezis2000, Brezis1997, Cianchi2009-2, delPino2010, Edmunds1999, Filippas2002, Gazzola2004, Leray1933, Opic1990, Sano2017,  Sano2022} and the references therein. For quantitative stability results in various Sobolev and Hardy-Sobolev inequalities, we refer to \cite{Debdip2026, Bianchi1991, Cianchi2006, Cianchi2009, Dolbeault2025, Figalli2019, Figalli2022,  Fusco2007, Fusco2008, Neumayer2020} and the references therein. The present work contributes to this line of research by establishing an improved quantitative stability estimate for the critical Hardy inequality using a completely rearrangement-free approach.

\smallskip

The structure of the paper is as follows:
\begin{itemize}
    \item[Section \ref{Section 1}:] Introduction, an overview of the critical Hardy inequality and its quantitative stability, together with the statement of the main theorems.

    \item[Section \ref{Section 2}:] Preliminaries on Orlicz spaces, the Luxemburg norm, and the integrability properties of the virtual extremizer. We also establish a critical Hardy inequality with a remainder term.

    \item[Section \ref{Section 3}:] Scale-invariant Sobolev inequalities with explicit dependence on the constants. We also recall an estimate relating the averages over two disjoint sets.

    \item[Section \ref{Section 4}:] Proof of the main theorems. Our proof is entirely rearrangement-free and is based on the scale-invariant Sobolev inequality, the remainder estimate for the critical Hardy inequality, and a careful summation argument over a dyadic decomposition.
\end{itemize}

\section{The Luxemburg norm, the virtual extremizer, and a remainder estimate}\label{Section 2}

In this section, we recall the Orlicz space and the associated Luxemburg norm used throughout the paper. We then establish several preliminary results concerning the virtual extremizer associated with the critical Hardy inequality. In particular, we prove its integrability properties, compute its Luxemburg norm explicitly, and establish a critical Hardy inequality with a remainder term, which will play a key role in the proof of our main theorems.

\smallskip

Let $\Omega \subset \mathbb{R}^{N}$ be a measurable set and let
$\Phi : [0,\infty) \rightarrow [0,\infty)$ be an $N$-function, that is,
$\Phi$ is continuous, convex, increasing,
$\Phi(0)=0$, and satisfies
\begin{equation*}
\lim_{t\to0^{+}}\frac{\Phi(t)}{t}=0,
\qquad
\lim_{t\to\infty}\frac{\Phi(t)}{t}=\infty.
\end{equation*}

The corresponding Orlicz space is defined by
\begin{equation*}
L^{\Phi}(\Omega)
:=
\left\{
u:\Omega\rightarrow\mathbb{R}\text{ measurable }:
\int_{\Omega}
\Phi\!\left(\frac{|u(x)|}{\lambda}\right)\,dx
<\infty
\text{ for some }\lambda>0
\right\}.
\end{equation*}
The Luxemburg norm on $L^{\Phi}(\Omega)$ is defined by
\begin{equation*}
\|u\|_{L^{\Phi}(\Omega)}
=
\inf
\left\{
\lambda>0:
\int_{\Omega}
\Phi\!\left(\frac{|u(x)|}{\lambda}\right)\,dx
\le1
\right\}.
\end{equation*}
It is well known that the Luxemburg norm is indeed a norm on
$L^{\Phi}(\Omega)$ and that $(L^{\Phi}(\Omega),\|\cdot\|_{L^{\Phi}(\Omega)})$
is a Banach space.

Let $\Phi(t) = \exp(t^{\frac{N}{N-1}}) -1$ for all $t \geq 0$. For simplicity, we denote $ L^{\Phi}(\Omega) = \operatorname{Exp} L^{\frac{N}{N-1}}(\Omega)$ and write the corresponding Luxemburg norm as
\begin{equation}\label{Defn: Exp Norm}
\|u\|_{L^{\Phi}(\Omega)}   =  \|u\|_{\operatorname{Exp} L^{\frac{N}{N-1}}(\Omega)}.
\end{equation}

\subsection{Integrability properties of the virtual extremizer}

Let $B_{1} = B_{1}(0)$ denote the unit ball in $\mathbb{R}^{N}$, where $N\geq 2$, consider the virtual extremizer associated with the critical Hardy inequality, given in \eqref{Defn: Omega}. Although $\omega_{R}$ is unbounded near the origin, it belongs to every
finite Lebesgue space. Moreover, it belongs to the critical exponential
Orlicz space. The following lemma gives the integrability properties of the virtual extremizer.

\begin{lemma}\label{Lemma: On virtual Extremizers}
Let $N\geq 2$, and $B_{1}= B_{1}(0)$ be a unit ball in $\mathbb{R}^{N}$. Let $\omega=\omega_{1}$ be defined in \eqref{Defn: Omega}. Then the following assertions hold.
\begin{enumerate}
    \item For every $q\in[1,\infty)$, $\omega\in L^{q}(B_{1}),$
    and
    \begin{equation*}
        \|\omega\|_{L^{q}(B_{1})}
        =
        \frac{|B_{1}|^{\frac{1}{q}}}{N^{\frac{N-1}{N}}}
        \left[
        \Gamma\left(
        1+\frac{N-1}{N}q
        \right)
        \right]^{\frac{1}{q}}.
    \end{equation*}

    \item Let  $ \Phi(t):= \exp\left(t^{\frac{N}{N-1}}\right)-1$, for  $t\geq 0$. Then $\omega\in
        \operatorname{Exp} L^{\frac{N}{N-1}}(B_{1})$, and its Luxemburg norm is given by
    \begin{equation*}
  \|\omega\|_{\operatorname{Exp} L^{\frac{N}{N-1}}(B_{1})}
        =
        \left(
        \frac{1+|B_{1}|}{N}
        \right)^{\frac{N-1}{N}}.
    \end{equation*}
\end{enumerate}
\end{lemma}
\begin{proof}
Let $q\in[1,\infty)$. Using polar coordinates, we obtain
\begin{align*}
    \int_{B_{1}}|\omega(x)|^{q}\,dx
    &=
    |\mathbb{S}^{N-1}|
    \int_{0}^{1}
    \left(
    \log\left(\frac{1}{r}\right)
    \right)^{\frac{N-1}{N}q}
    r^{N-1}\,dr.
\end{align*}
We make the change of variables $t=\log\left(\frac{1}{r}\right)$. Then $r=e^{-t}$,  $dr=-e^{-t}\,dt$,
and hence $r^{N-1}\,dr=-e^{-Nt}\,dt$. Therefore,
\begin{align*}
    \int_{B_{1}}|\omega(x)|^{q}\,dx
    &=
    |\mathbb{S}^{N-1}|
    \int_{0}^{\infty}
    t^{\frac{N-1}{N}q}e^{-Nt}\,dt.
\end{align*}
Using the Gamma function identity $\int_{0}^{\infty}t^{\alpha}e^{-Nt}\,dt   = \frac{\Gamma(\alpha+1)} {N^{\alpha+1}}$, for $\alpha>-1$, with $\alpha=\frac{N-1}{N}q$, we obtain
\begin{align*}
\int_{B_{1}}|\omega(x)|^{q}\,dx
    =
    \frac{|\mathbb{S}^{N-1}|}
    {N^{1+\frac{N-1}{N}q}}
    \Gamma\left(
    1+\frac{N-1}{N}q
    \right).
\end{align*}
Hence, $\omega\in L^{q}(B_{1})$ for every $q\in[1,\infty)$.

We next prove the exponential integrability of $\omega$. By the
definition of the Luxemburg norm,
\begin{align*}
    \|\omega\|_{\operatorname{Exp} L^{\frac{N}{N-1}}(B_{1})}
    =
    \inf\Bigg\{
    \lambda>0:
    \int_{B_{1}}
    \left[
    \exp\left(
    \left(
    \frac{\omega(x)}{\lambda}
    \right)^{\frac{N}{N-1}}
    \right)-1
    \right]dx
    \leq 1
    \Bigg\}.
\end{align*}
Since $\omega(x)^{\frac{N}{N-1}} = \log 1/|x|$, we have
\begin{align*}
    \int_{B_{1}}
    \left[
    \exp\left(
    \left(
    \frac{\omega(x)}{\lambda}
    \right)^{\frac{N}{N-1}}
    \right)-1
    \right]dx
    &=
    \int_{B_{1}}
    \left[
    \exp\left(
    \lambda^{-\frac{N}{N-1}}
    \log\left(\frac{1}{|x|}\right)
    \right)-1
    \right]dx
    \\
    & =
    \int_{B_{1}}
    \left(
    |x|^{-\lambda^{-\frac{N}{N-1}}}-1
    \right)dx =  \frac{
    |B_{1}|\lambda^{-\frac{N}{N-1}}
    }{
    N-\lambda^{-\frac{N}{N-1}}
    }.
\end{align*}
Thus, the condition in the definition of the Luxemburg norm is equivalent to $ \lambda \geq \left( \frac{1+|B_{1}|}{N} \right)^{\frac{N-1}{N}}$. Therefore,
\begin{equation*}
\|\omega\|_{\operatorname{Exp} L^{\frac{N}{N-1}}(B_{1})}
    =
    \left(
    \frac{1+|B_{1}|}{N}
    \right)^{\frac{N-1}{N}}.
\end{equation*}
In particular, $\omega\in \operatorname{Exp} L^{\frac{N}{N-1}}(B_{1})$. This completes the proof.
\end{proof}

The next lemma establishes a critical Hardy inequality with a remainder term, which serves as a key ingredient in the proof of our main theorems. The proof relies on the fact that the virtual extremizer $\omega_{R}(x)
= \left( \log\frac{R}{|x|}
\right)^{\frac{N-1}{N}}$ satisfies the Euler--Lagrange equation associated with the critical Hardy inequality,
\begin{equation}\label{Euler-Lagrange Equation}
-\operatorname{div}\!\left(|\nabla\omega_{R}|^{N-2}\nabla\omega_{R}\right)
=
\left(\frac{N-1}{N}\right)^{N}
\frac{(\omega_{R}(x))^{N-1}}
{|x|^{N}\left(\log\frac{R}{|x|}\right)^{N}}.
\end{equation}
Combining this identity with the following convexity estimate, valid for all
$X,Y\in\mathbb{R}^{N}$ with $N\geq2$ (see \cite[Inequality $(2.13)$]{Frank2008}),
\begin{equation}\label{Convex Estimate}
|X+Y|^{N}
\geq
|X|^{N}
+
N|X|^{N-2}X\cdot Y
+
c_{N}|Y|^{N},
\end{equation}
where $c_{N}:= \min_{0 < \tau<1/2} \left( (1-\tau)^{N}  - \tau^{N}  + N \tau^{N-1} \right) $, we obtain the following remainder estimate.

\begin{lemma}\label{Lemma: Critical Hardy remainder}
Let $\Omega$ be a bounded domain containing origin in $\mathbb{R}^{N}$, where $N \geq 2$, and let $\widetilde{R} = \sup_{x \in \Omega}$ and $R \geq \widetilde{R}$. Then for any $u \in W^{1,N}_{0}(\Omega)$,
\begin{align}\label{Hardy inequality with remainder p=N} \int_{\Omega} |\nabla u(x)|^{N} \, dx - & \left( \frac{N-1}{N} \right)^{N} \bigintsss_{\Omega} \frac{|u(x)|^{N}}{|x|^{N} \left( \log\frac{R}{|x|} \right)^{N}} \, dx \nonumber \geq c_{N} \int_{\Omega} |\nabla v(x)|^{N} \left( \log \frac{R}{|x|} \right)^{N-1} \, dx , 
\end{align}
where $v(x) = u(x) \omega^{-1}_{R}(x)$. \end{lemma}
\begin{proof}
By density, it is sufficient to prove the result for $u\in C_{c}^{\infty}(\Omega)$. Let $u=v\omega_{R}$, where $\omega_{R}$ is defined in
\eqref{Defn: Omega}. Then
    \begin{equation*}
        \nabla u = v \nabla \omega_{R} + \omega_{R} \nabla v. 
    \end{equation*}
Applying the convexity estimate \eqref{Convex Estimate} with $X=v\nabla\omega_{R}$ and $Y=\omega_{R}\nabla v$, we obtain
    \begin{align*}
   \int_{\Omega} & |\nabla u(x)|^{N} \, dx   =  \int_{\Omega} \left|X+Y \right|^{N} \, dx \\ &\geq \left( \frac{N-1}{N} \right)^{N} \bigintsss_{\Omega} \frac{|u(x)|^{N}}{|x|^{N} \left( \log\frac{R}{|x|} \right)^{N}} \, dx + N \int_{\Omega} |\nabla \omega_{R}|^{N-2} \omega_{R} v |v|^{N-2} \nabla v \cdot \nabla \omega_{R} \, dx \\ & \quad \quad +  c_{N} \int_{\Omega} |\nabla v(x)|^{N} \left( \log \frac{R}{|x|}  \right)^{N-1} \, dx .
\end{align*}
Observe that
\begin{equation*}
N |\nabla \omega_{R}|^{N-2}\omega_{R}v|v|^{N-2}
\nabla v\cdot\nabla\omega_{R}
=
\omega_{R}|\nabla\omega_{R}|^{N-2}
\nabla\omega_{R}\cdot\nabla(|v|^{N}).
\end{equation*}
Hence, integrating by parts, we obtain 
\begin{align*}
   N \int_{\Omega} |\nabla \omega_{R}|^{N-2} \omega_{R} v |v|^{N-2} \nabla v \cdot \nabla \omega_{R} \, dx & =  \int_{\Omega} \omega_{R}|\nabla\omega_{R}|^{N-2}
\nabla\omega_{R}\cdot\nabla(|v|^{N}) \, dx \\ & = - \int_{\Omega} |\nabla \omega_{R}|^{N} |v|^{N} \, dx \\ &  \quad  - \int_{\Omega} \omega_{R} \operatorname{div} \left( |\nabla \omega_{R}|^{N-2} \nabla \omega_{R} \right) |v|^{N} \, dx  . 
\end{align*}
Using the Euler--Lagrange equation \eqref{Euler-Lagrange Equation}, we have $-\omega_{R} \operatorname{div} \left( |\nabla \omega_{R}|^{N-2} \nabla \omega_{R} \right) = |\nabla \omega_{R}|^{N}$. Therefore, the above term vanishes. Hence,
\begin{align*}\int_{\Omega} |\nabla u(x)|^{N} \, dx - & \left( \frac{N-1}{N} \right)^{N} \bigintsss_{\Omega} \frac{|u(x)|^{N}}{|x|^{N} \left( \log\frac{R}{|x|} \right)^{N}} \, dx \nonumber \geq c_{N} \int_{\Omega} |\nabla v(x)|^{N} \left( \log \frac{R}{|x|} \right)^{N-1} \, dx . 
\end{align*}
This completes the proof of lemma.
\end{proof}

\section{Sobolev Inequalities and Auxiliary Estimates}\label{Section 3}

In this section, we establish two preliminary results that are essential for our rearrangement-free approach. We first prove a scale-independent Sobolev
inequality in the critical case $p=N$, where the constant is independent of the scaling parameter and has explicit dependence on the exponent $q$. This estimate plays a key role in the annular decomposition and the summation arguments used later. We also prove an elementary estimate comparing the
averages of a function over two disjoint sets, which will be repeatedly used in the proof of the main theorem.

\smallskip

The next lemma establishes a scale-independent Sobolev inequality in the critical case $p=N$. The key feature of this inequality is that the constant is independent of the scaling parameter $\lambda$, while its dependence on the exponent $q$ is explicitly given by $q^{\frac{N-1}{N}}$. This explicit dependence will play an important role in the summation arguments used later. For any bounded open set $\Omega$, 
\begin{equation*}
(u)_{\Omega} := \frac{1}{|\Omega|} \int_{\Omega} u(x) \, dx  = \fint_{\Omega} u(x) \, dx 
\end{equation*}
denotes the average value of $u$ over $\Omega$, where $|\Omega|$ is the Lebesgue measure of $\Omega$.

\begin{lemma}\label{Lemma: Sobolev inequality}
Let $\Omega$ be a bounded $C^{1}$ domain in $\mathbb{R}^{N}$. For $\lambda>0$, define $\Omega_{\lambda}:=\{\lambda x:x\in\Omega\}$. Then there exists a constant $C=C(N,\Omega)>0$ such that for every $q>N$,
\begin{equation}
    \left(
\fint_{\Omega_{\lambda}}
|u(x)-(u)_{\Omega_{\lambda}}|^{q}\,dx
\right)^{\frac1q}
\leq
C\,q^{\frac{N-1}{N}}
\left(
\int_{\Omega_{\lambda}}
|\nabla u(x)|^{N}\,dx
\right)^{\frac1N},
\end{equation}
for every $u\in W^{1,N}(\Omega_{\lambda})$.
\end{lemma}
\begin{proof}
By \cite[Theorem 12.33]{LeoniBook}, for every $q>N$, there exists a constant
$C=C(N)>0$ such that
\begin{equation*}
    \|u\|_{L^{q}(\mathbb{R}^{N})}
\leq
C\,q^{\frac{N-1}{N}}
\left(
\|u\|_{L^{N}(\mathbb{R}^{N})}^{N}
+
\|\nabla u\|_{L^{N}(\mathbb{R}^{N})}^{N}
\right)^{\frac{1}{N}},
\end{equation*}
for every $u\in W^{1,N}(\mathbb{R}^{N})$. Since $\Omega$ is a bounded $C^{1}$ domain, it is a uniformly bounded Lipschitz domain. Hence, by the Sobolev extension theorem
(\cite[Theorem 13.17]{LeoniBook}), there exists a bounded linear extension operator $E:W^{1,N}(\Omega)\rightarrow W^{1,N}(\mathbb{R}^{N})$, such that $Eu=u$ a.e. in $\Omega$ and
\begin{equation*}
    \|Eu\|_{L^{N}(\mathbb{R}^{N})}^{N}
+
\|\nabla Eu\|_{L^{N}(\mathbb{R}^{N})}^{N}
\leq
C
\left(
\|u\|_{L^{N}(\Omega)}^{N}
+
\|\nabla u\|_{L^{N}(\Omega)}^{N}
\right),
\end{equation*}
where $C=C(N,\Omega)>0$. Combining the above two inequalities yields
\begin{equation*}
   \|u\|_{L^{q}(\Omega)}
\leq
C\,q^{\frac{N-1}{N}}
\left(
\|u\|_{L^{N}(\Omega)}^{N}
+
\|\nabla u\|_{L^{N}(\Omega)}^{N}
\right)^{\frac{1}{N}}, 
\end{equation*}
for every $u\in W^{1,N}(\Omega)$. Next, by the Poincar\'e inequality
(\cite[Section 5.8.1, Theorem 1]{EvansBook}), the above inequality reduces to 
\begin{equation*}
    \|u-(u)_{\Omega}\|_{L^{q}(\Omega)}
\leq
C\,q^{\frac{N-1}{N}}
\|\nabla u\|_{L^{N}(\Omega)},
\end{equation*}
where $C=C(N,\Omega)>0$. Now apply the above inequality to the function $u(\lambda x)$. Using the identity $\fint_{\Omega} u(\lambda x)\,dx = \fint_{\Omega_{\lambda}} u(x)\,dx$, multiplying both sides of the inequality by $\frac{1}{|\Omega|^{1/q}}$, and readjusting the constant $C=C(N,\Omega)>0$, we obtain
\begin{equation*}
    \left(
\fint_{\Omega}
|u(\lambda x)-(u)_{\Omega_{\lambda}}|^{q}\,dx
\right)^{\frac{1}{q}}
\leq
C\,q^{\frac{N-1}{N}}
\left(
\lambda^{N}
\int_{\Omega}
|\nabla u(\lambda x)|^{N}\,dx
\right)^{\frac{1}{N}},
\end{equation*}
Finally, making the change of variables $y=\lambda x$, we obtain
\begin{equation*}
  \left(
\fint_{\Omega_{\lambda}}
|u(y)-(u)_{\Omega_{\lambda}}|^{q}\,dy
\right)^{\frac{1}{q}}
\leq
C\,q^{\frac{N-1}{N}}
\left(
\int_{\Omega_{\lambda}}
|\nabla u(y)|^{N}\,dy
\right)^{\frac{1}{N}},  
\end{equation*}
which completes the proof.
\end{proof}

The next lemma provides an estimate for the difference between the averages of a function over two disjoint sets. This estimate plays an important role in comparing the averages of $u$ on adjacent annuli during the summation
arguments.

\begin{lemma}\label{Lemma: on two disjoint set}
    Let $E$ and $F$ be two disjoint sets in $\mathbb{R}^{N}$ and $q \geq 1$. Then 
    \begin{equation}
        |(u)_{E} - (u)_{F}|^{q} \leq 2^{q} \frac{|E \cup F|}{\min \{ |E|, |F| \} }  \fint_{E \cup F} |u(x)-(u)_{E \cup F}|^{q} \, dx  . 
    \end{equation}
\end{lemma}
\begin{proof} 
We begin by estimating $|(u)_E-(u)_F|^q$. Using the triangle inequality, we obtain
\begin{equation*}
\begin{split}
    |(u)_{E}-(u)_{F}|^{q}  & \leq 2^{q} |(u)_{E} - (u)_{E \cup F}|^{q} +  2^{q} | (u)_{F} + (u)_{E \cup F}|^{q} \\
       & = 2^{q} \Big| \fint_{E} \left\{ u(x) - (u)_{E \cup F} \right\}   dx  \Big|^{q} + 2^{q} \Big| \fint_{F} \left\{ u(x) - (u)_{E \cup F} \right\}  dx \Big|^{q}. 
\end{split}
\end{equation*}
    By using H$\ddot{\text{o}}$lder's inequality with ~$ \frac{1}{q} +  \frac{1}{q'} = 1$, we have
    \begin{equation*}
    \begin{split}
     |(u)_{E}-(u)_{F}|^{q}  &\leq  2^{q} \fint_{E} |u(x) - (u)_{E \cup F} |^{q}   dx + 2^{q} \fint_{F} | u(x) - (u)_{E \cup F} |^{q}  dx \\
       &\leq  \frac{2^{q}}{\text{min} \{ |E|, |F| \} } \int_{E \cup F} |u(x) - (u)_{E \cup F} |^{q}   dx \\&
       = 2^{q} \frac{|E \cup  F|}{\text{min} \{ |E|, |F| \}} \fint_{E \cup F} |u(x) - (u)_{E \cup F} |^{q}   dx .
    \end{split}
    \end{equation*}
    This completes the proof.
\end{proof}

\section{Proof of the Main Theorems}\label{Section 4}

In this section, we prove Theorem~\ref{Theorem: Quantitative stability critical Hardy} and Theorem~\ref{Theorem 2}, establishing a new quantitative stability estimate for the critical Hardy inequality and an improvement of the stability estimate \eqref{Cianchi Quantitative} due to Cianchi and Ferone \cite[Theorem~1.2]{CianchiFerone}. Our proof is completely rearrangement-free and improves the result of Cianchi and Ferone \cite[Theorem $1.2$]{CianchiFerone}. The proof is based on the scale-independent Sobolev inequality established in the previous section together with suitable summation arguments, which naturally lead to the introduction of the distance function. Throughout this section, we use the following notation:
\begin{equation}\label{Defn: Phi}
    \Phi(t) := \exp \left( t^{\frac{N}{N-1}} \right) -1, \quad \forall \, t \geq 0. 
\end{equation}

\smallskip

\begin{proof}[\textbf{Proof of Theorem \ref{Theorem: Quantitative stability critical Hardy}}]
Let $\Omega$ be a bounded open set containing the origin, and let
$\widetilde{R}=\sup_{x\in\Omega}|x|$ with $R>\widetilde{R}$. We first prove the improved quantitative stability estimate for functions
$u\in C_{c}(\Omega) \cap W^{1,N}_{0}(\Omega)$ with $u\geq0$. By extending $u$ by zero outside $\Omega$, we may regard $u$ as an element of
$C_{c}(B_{\widetilde{R}}(0)) \cap W^{1,N}_{0}(\Omega)$. Finally, we extend the result to every $u\in W_{0}^{1,N}(\Omega)$ by a standard density argument.

Define the Hardy deficit by
\begin{equation}\label{Defn: Hardy deficit}
    \mathcal{H}_{N,R}(u)
    :=
    \int_{\Omega} |\nabla u(x)|^{N}\,dx
    -
    \left(\frac{N-1}{N}\right)^{N}
    \bigintsss_{\Omega}
    \frac{|u(x)|^{N}}
    {|x|^{N}\left( \log\frac{R}{|x|} \right)^{N}}
    \,dx.
\end{equation}

We divide the proof into two cases: $\mathcal{H}_{N,R}(u)\leq1$
and $\mathcal{H}_{N,R}(u)>1$. In the second case, we normalize the
function by assuming $\|u\|_{\operatorname{Exp}L^{N/(N-1)}(\Omega)} \leq 1$.
This normalization is not required in the first case. Therefore, the argument
for the first case also remains valid under the above normalization.

\subsection{The Case \texorpdfstring{$\mathcal{H}_{N,R}(u)\leq 1$}{delta leq 1}:} In this subsection, we assume that
$\mathcal{H}_{N,R}(u)\leq1$. Since
$u\in C_{c}(B_{\widetilde{R}}(0)) \cap W^{1,N}_{0}(\Omega)$ and $u\geq0$, we define the set
\begin{equation*}
    \mathcal{G} := \{ x \in \Omega : v(x) > (\mathcal{H}_{N,R}(u))^{\frac{1}{N}} \},
\end{equation*}
where $v= u \omega^{-1}_{R}$, and $\omega_{R}(x) = \left( \log R/|x| \right)^{(N-1)/N}$. For any $\ell \in \mathbb{Z}$ and $\ell \leq -1$, define 
\begin{equation*}
    A_{\ell} := \{ x \in \mathbb{R}^{N} : 2^{\ell} R \leq |x| < 2^{\ell+1} R \}.
\end{equation*}
Since $\lim_{|x| \to 0} v(x) = u(x) \omega^{-1}_{R}(x) = 0$, choose the smallest integer
$\ell_{0}$ such that
\begin{equation*}
    v(x) \leq (\mathcal{H}_{N,R}(u))^{\frac{1}{N}}, \quad  \forall \, x \in A_{\ell_{0}} \quad \text{and} \quad \mathcal{G} \cap A_{\ell_{0}+1} \neq \phi \quad \Rightarrow (v)_{A_{\ell_{0}}} \leq (\mathcal{H}_{N,R}(u))^{\frac{1}{N}}. 
\end{equation*}
Therefore, using the inequality $0< v(x) - (\mathcal{H}_{N,R}(u))^{\frac{1}{N}} \leq v(x) -(v)_{A_{\ell_{0}}}$, for all $x \in \mathcal{G}$, and $q>N$, we obtain
\begin{align}\label{ineq4}
    \int_{\mathcal{G}} |v(x) - (\mathcal{H}_{N,R}(u))^{\frac{1}{N}}|^{q}  \left( \omega_{R}(x) \right)^{q} \,  dx  & \leq \int_{\mathcal{G}} |v(x) -(v)_{A_{\ell_{0}}}|^{q} \left( \omega_{R}(x) \right)^{q} \, dx \nonumber  \\ & \leq  \sum_{\substack{\ell \leq -1 \\ \mathcal{G} \cap A_{\ell} \neq \phi}} \int_{A_{\ell}} |v(x) -(v)_{A_{\ell_{0}}}|^{q} \left( \omega_{R}(x) \right)^{q} \, dx,
\end{align}
where $\omega_{R}$ is defined in \eqref{Defn: Omega}. Let $\ell\leq -1$ be such that $\mathcal{G}\cap A_{\ell}\neq\emptyset$. For
any $x\in A_{\ell}$, we have $2^{\ell}R\leq |x|$, which implies that $\log R/|x|  \leq (-\ell) \log 2$. Hence,  $\omega_{R}(x) \leq \left( (-\ell) \log 2 \right)^{(N-1)/N}$.  Therefore,
\begin{align}\label{ineqn1}
  \int_{A_{\ell}} |v(x) -(v)_{A_{\ell_{0}}}|^{q}  \left( \omega_{R}(x) \right)^{q} & \, dx  \leq ((-\ell) \log 2)^{\frac{N-1}{N} q} \int_{A_{\ell}} |v(x) -(v)_{A_{\ell_{0}}}|^{q} \, dx \nonumber \\ & \leq 2^{q} ((-\ell) \log 2)^{\frac{N-1}{N} q}   \int_{A_{\ell}} |v(x) -(v)_{A_{\ell}}|^{q} \, dx \nonumber \\ & \quad + 2^{q} ((-\ell) \log 2)^{\frac{N-1}{N} q}  |(v)_{A_{\ell}} - (v)_{A_{\ell_{0}}}|^{q}  |A_{\ell}|,
\end{align}
where $|A_{\ell}|$ denotes the Lebesgue measure of $A_{\ell}$. Applying Lemma \ref{Lemma: Sobolev inequality} with $\Omega = \{ x \in \mathbb{R}^{N} : R<|x|< 2R\}$ and $\lambda = 2^{\ell}$, we obtain
\begin{align}\label{ineq6}
    ((-\ell) \log 2)^{\frac{N-1}{N} q}  & \int_{A_{\ell}} |v(x) -(v)_{A_{\ell}}|^{q} \, dx  = ((-\ell) \log 2)^{\frac{N-1}{N} q}   |A_{\ell}| \fint_{A_{\ell}} |v(x) -(v)_{A_{\ell}}|^{q} \, dx \nonumber \\ & \leq C^{q} q^{\frac{N-1}{N} q} ((-\ell) \log 2)^{\frac{N-1}{N} q} |A_{\ell}|  \left( \int_{A_{\ell}} |\nabla v(x)|^{N} \, dx \right)^{\frac{q}{N}},
\end{align}
where $C=C(N,R)>0$. Now let $\ell\leq -2$. For any $x\in A_{\ell}$, we have $(-\ell-1) \log 2 < \log R/|x| $.  Hence, using the inequality $(-\ell) \leq 2(-\ell-1)$ for all $\ell\leq-2$,
together with the definition of $\omega_{R}$ in \eqref{Defn: Omega} and the
fact that $\operatorname{supp} v\subset\Omega$, we obtain
\begin{align*}
    ((-\ell) \log 2)^{\frac{N-1}{N} q}    \int_{A_{\ell}} |v(x) -(v)_{A_{\ell}}|^{q} \, dx  & \leq C^{q} q^{\frac{N-1}{N} q} |A_{\ell}| \left( \int_{A_{\ell}} |\nabla v(x)|^{N} \left( \log  \frac{R}{|x|} \right)^{N-1} \, dx \right)^{\frac{q}{N}} \\ & = C^{q} q^{\frac{N-1}{N} q} |A_{\ell}| \left( \int_{A_{\ell}} |\nabla v(x)|^{N} \left( \omega_{R}(x) \right)^{N} \, dx \right)^{\frac{q}{N}} \\ & \leq C^{q} q^{\frac{N-1}{N} q} |A_{\ell}| \left( \int_{\Omega} |\nabla v(x)|^{N} \left( \omega_{R}(x) \right)^{N} \, dx \right)^{\frac{q}{N}}.
\end{align*}
Consider the case $\ell=-1$. Since $\operatorname{supp} v\subset\Omega$,
$R>\widetilde{R}$, and $\widetilde{R} = \sup_{x \in \Omega} |x|$, we have $0<\log R/\widetilde{R} \leq \log R/|x|$ for all $x \in \Omega$. Therefore, applying \eqref{ineq6} and using the definition of $\omega_{R}$, we obtain
\begin{align*}
  (\log 2)^{\frac{N-1}{N} q}   \int_{A_{-1}} & |v(x) -(v)_{A_{-1}}|^{q} \, dx  \leq C^{q} q^{\frac{N-1}{N} q} (\log 2)^{\frac{N-1}{N} q} |A_{-1}|  \left( \int_{\Omega} |\nabla v(x)|^{N} \, dx \right)^{\frac{q}{N}} \\ & \leq C^{q} q^{\frac{N-1}{N} q} \left( \frac{\log 2}{  \log \frac{R}{\widetilde{R}}} \right)^{\frac{N-1}{N} q } |A_{-1}| \left( \int_{\Omega} |\nabla v(x)|^{N} \left( \omega_{R}(x) \right)^{N} \, dx \right)^{\frac{q}{N}}. 
\end{align*}
Combining this estimate with the case $\ell\leq-2$, we obtain, for every
$\ell\leq-1$,
\begin{equation*}
    ((-\ell) \log 2)^{\frac{N-1}{N} q}    \int_{A_{\ell}} |v(x) -(v)_{A_{\ell}}|^{q} \, dx \leq C^{q} q^{\frac{N-1}{N} q} |A_{\ell}| \left( \int_{\Omega} |\nabla v(x)|^{N} \left( \omega_{R}(x) \right)^{N} \, dx \right)^{\frac{q}{N}},
\end{equation*}
where $C=C(N,R,\widetilde{R})>0$ is independent of $\ell$. Therefore, using $\sum_{\ell \leq -1} |A_{\ell}| \leq |B_{R}(0)|$, we obtan
\begin{align}\label{ineq2}
  \sum_{\ell \leq -1} ((-\ell) \log 2)^{\frac{N-1}{N} q}    \int_{A_{\ell}} |v(x) -(v)_{A_{\ell}}|^{q} \, dx  \leq C^{q} q^{\frac{N-1}{N} q} \left( \int_{\Omega} |\nabla v(x)|^{N} \left( \omega_{R}(x) \right)^{N} \, dx \right)^{\frac{q}{N}},  
\end{align}
where $C=C(N,R,\widetilde{R})>0$ is a uniform constant. Next, we apply Lemma~\ref{Lemma: on two disjoint set} with
$E=A_{k}$ and $F=A_{k+1}$. Then, applying
Lemma~\ref{Lemma: Sobolev inequality} with $\Omega = \{ x \in \mathbb{R}^{N} : R< |x|<4R \}$ and $\lambda =2^{k}$, we obtain
\begin{align}\label{ineq7}
    ((-\ell) \log 2)^{\frac{N-1}{N} q} & |(v)_{A_{\ell}} - (v)_{A_{\ell_{0}}}|^{q}  \leq  ((-\ell) \log 2)^{\frac{N-1}{N} q}  2^{q} \sum_{k=\ell_{0}}^{\ell-1}  |(v)_{A_{k}} - (v)_{A_{k+1}}|^{q} \nonumber \\ & \leq C^{q} q^{\frac{N-1}{N}q}  ((-\ell) \log 2)^{\frac{N-1}{N} q} \sum_{k=\ell_{0}}^{\ell-1} \left( \int_{A_{k} \cup A_{k+1}} |\nabla v(x)|^{N} \, dx \right)^{\frac{q}{N}}.
\end{align}
For $\ell\leq-1$ and $\ell_{0}\leq k<\ell$, excluding the case $\ell=-1$ and $k=-2$, let $x\in A_{k}\cup A_{k+1}$. Then $ (-k-2) \log 2 \leq \log R/|x| $. Hence, using the inequality $\frac{(-\ell) \log 2}{(-k-2) \log 2} \leq 2$ , together with the definition of $\omega_{R}$ in
\eqref{Defn: Omega}, the estimate above yields
\begin{align*}
    ((-\ell) \log 2)^{\frac{N-1}{N} q}  & |(v)_{A_{\ell}} - (v)_{A_{\ell_{0}}}|^{q}  \\ & \leq C^{q} q^{\frac{N-1}{N}q} \sum_{k=\ell_{0}}^{\ell-1} \left( \int_{A_{k} \cup A_{k+1}} |\nabla v(x)|^{N} \left( \log \frac{R}{|x|}  \right)^{N-1} \, dx \right)^{\frac{q}{N}} \\ & \leq C^{q} q^{\frac{N-1}{N}q} \left( \int_{\Omega} |\nabla v(x)|^{N} \left( \omega_{R}(x) \right)^{N} \, dx \right)^{\frac{q}{N}} . 
\end{align*}
Now consider the exceptional case $\ell=-1$ and $k=-2$. Since
$\operatorname{supp} v \subset\Omega$ and  $0<\log R/\widetilde{R} \leq \log R/|x|$ for all $x \in \Omega$, the right-hand side of \eqref{ineq7} satisfies
\begin{align*}
    C^{q} q^{\frac{N-1}{N}q}  (\log 2)^{\frac{N-1}{N} q} &  \left( \int_{A_{-2} \cup A_{-1}} |\nabla v(x)|^{N} \, dx \right)^{\frac{q}{N}} \\ & \leq C^{q} q^{\frac{N-1}{N}q}  \left(\frac{\log 2}{\log \frac{R}{\widetilde{R}}} \right)^{\frac{N-1}{N} q}  \left( \int_{\Omega} |\nabla v(x)|^{N} \left( \log \frac{R}{|x|}  \right)^{N-1} \, dx \right)^{\frac{q}{N}} \\ & \leq C^{q} q^{\frac{N-1}{N}q} \left( \int_{\Omega} |\nabla v(x)|^{N} \left( \omega_{R}(x) \right)^{N} \, dx \right)^{\frac{q}{N}},
\end{align*}
where $C=C(N,R,\widetilde{R})>0$ is a uniform constant. Combining the above estimates with \eqref{ineq7}, we conclude that for any $\ell \leq -1$,
\begin{align*}
    ((-\ell) \log 2)^{\frac{N-1}{N} q} & |(v)_{A_{\ell}} - (v)_{A_{\ell_{0}}}|^{q} \leq C^{q} q^{\frac{N-1}{N}q} \left( \int_{\Omega} |\nabla v(x)|^{N} \left( \omega_{R}(x) \right)^{N} \, dx \right)^{\frac{q}{N}}.
\end{align*}
Therefore, using $\sum_{\ell \leq -1} |A_{\ell}| \leq |B_{R}(0)|$, we obtan
\begin{align}\label{ineq3}
  \sum_{\ell \leq -1}  ((-\ell) \log 2)^{\frac{N-1}{N} q} &  |(v)_{A_{\ell}} - (v)_{A_{\ell_{0}}}|^{q}   |A_{\ell}| \nonumber \\  & \leq  C^{q} q^{\frac{N-1}{N}q} \left( \sum_{\ell \leq -1} |A_{\ell}| \right) \left( \int_{\Omega} |\nabla v(x)|^{N} \left( \omega_{R}(x) \right)^{N} \, dx \right)^{\frac{q}{N}} \nonumber \\ & \leq  C^{q} q^{\frac{N-1}{N}q} \left( \int_{\Omega} |\nabla v(x)|^{N} \left( \omega_{R}(x) \right)^{N} \, dx \right)^{\frac{q}{N}}  ,
\end{align}
where $C=C(N,R,\widetilde{R})>0$ is a uniform constant. Combining the inequalities \eqref{ineqn1}, \eqref{ineq2}, and \eqref{ineq3}, we obtain
\begin{align*}
  \sum_{\ell \leq -1}  \int_{A_{\ell}}  |v(x) -(v)_{A_{\ell_{0}}}|^{q} \left( \omega_{R}(x) \right)^{q}  \, dx   \leq C^{q} q^{\frac{N-1}{N}q} \left( \int_{\Omega} |\nabla v(x)|^{N} \left( \omega_{R}(x) \right)^{N} \, dx \right)^{\frac{q}{N}} .
\end{align*}
Therefore, using the above inequality and the Hardy inequality with a remainder (see Lemma \ref{Lemma: Critical Hardy remainder}), the inequality \eqref{ineq4} yields
\begin{align}\label{eq:highq}
\int_{\mathcal{G}}
|v(x)-(\mathcal{H}_{N,R}(u))^{\frac{1}{N}}|^q
\left(\omega_{R}(x)\right)^{q}
\, dx
&\leq
C^q q^{\frac{N-1}{N}q}
\left(
\int_{\Omega}
|\nabla v(x)|^N
\left(\omega_{R}(x)\right)^{N}
\,dx
\right)^{\frac{q}{N}} \nonumber \\ & \leq
C^q q^{\frac{N-1}{N}q}
(\mathcal{H}_{N,R}(u))^{\frac{q}{N}},
\end{align}
for every $q>N$, where $C=C(N, R, \widetilde{R})>0$ is independent of $q$.

We now prove that the above estimate also holds for $1\le q\le N$.
Let
\[
F(x):=
\left|v(x)-(\mathcal{H}_{N,R}(u))^{\frac{1}{N}}\right|\omega_{R}(x).
\]
Since $N+1>N$, by \eqref{eq:highq} with $q=N+1$, we obtain
\begin{equation}
\int_{\mathcal G}|F(x)|^{N+1}\,dx
\le
C^{N+1}(N+1)^{\frac{N-1}{N}(N+1)}
(\mathcal{H}_{N,R}(u))^{\frac{N+1}{N}}.
\label{eq:Nplus1}
\end{equation}
Now let $1\leq q \leq N$. We apply H\"older's inequality with the conjugate exponents $r=\frac{N+1}{q}$, $r'=\frac{N+1}{N+1-q}$. Hence,
\begin{align*}
\int_{\mathcal G}|F(x)|^q\,dx
 \leq
\left(
\int_{\mathcal G}
|F(x)|^{N+1}\,dx
\right)^{\frac{q}{N+1}}
|\mathcal G|^{1-\frac{q}{N+1}}.
\end{align*}
Using \eqref{eq:Nplus1}, we obtain
\begin{align*}
\int_{\mathcal G}|F(x)|^q\,dx
&\le
\left[
C^{N+1}(N+1)^{\frac{N-1}{N}(N+1)}
(\mathcal{H}_{N,R}(u))^{\frac{N+1}{N}}
\right]^{\frac{q}{N+1}}
|\mathcal G|^{1-\frac{q}{N+1}}\\
&=
C^q
(N+1)^{\frac{N-1}{N}q}
(\mathcal{H}_{N,R}(u))^{\frac{q}{N}}
|\mathcal G|^{1-\frac{q}{N+1}}.
\end{align*}
Since $|\mathcal G|
\leq |B_{R}(0)|$, the quantity $(N+1)^{\frac{N-1}{N}q}
|\mathcal G|^{1-\frac{q}{N+1}}$ is bounded by $C^q$, where $C=C(N,R, \widetilde{R})>0$. Therefore,
\begin{equation*}
   \int_{\mathcal G}
|v(x)-(\mathcal{H}_{N,R}(u))^{\frac{1}{N}}|^q
\left( \omega_{R}(x) \right)^{q}
\,dx
\leq
C^{q}
(\mathcal{H}_{N,R}(u))^{\frac{q}{N}}, 
\end{equation*}
for every $1\le q\le N$. Combining this estimate with \eqref{eq:highq}, after enlarging the constant if necessary, we conclude that
\begin{equation*}
   \int_{\mathcal G}
|v(x)-(\mathcal{H}_{N,R}(u))^{\frac{1}{N}}|^{q}
\left(\omega_{R}(x)\right)^{q}
\,dx
\leq
C^q
q^{\frac{N-1}{N}q}
(\mathcal{H}_{N,R}(u))^{\frac{q}{N}}, 
\end{equation*}
for every $q \geq 1$.
Since $\mathcal{H}_{N,R}(u) \leq 1$, we have $(\mathcal{H}_{N,R}(u))^{\frac{q}{N}} \leq (\mathcal{H}_{N,R}(u))^{\frac{1}{N}}$. Therefore, using this and the definition of $v=u \omega^{-1}_{R}$ in the above inequality, we get
\begin{equation*}
    \int_{\mathcal{G}}  |u(x) - (\mathcal{H}_{N,R}(u))^{\frac{1}{N}} \omega_{R}(x)|^{q}  \,  dx \leq \mathcal{C}^{q} q^{\frac{N-1}{N}q} (\mathcal{H}_{N,R}(u))^{\frac{1}{N}}, 
\end{equation*}
where $\mathcal{C}=\mathcal{C}(N, R, \widetilde{R})>0$.

Now, for $\alpha>0$, using $q= \frac{nN}{N-1} \geq 1$, where $n \in \mathbb{N}$, we have
\begin{align*}
    \int_{\mathcal{G}} &  \left( \exp \left( \frac{\alpha}  {\mathcal{C}^{\frac{N}{N-1}}} |u(x) - (\mathcal{H}_{N,R}(u))^{\frac{1}{N}} \omega_{R}(x)|^{\frac{N}{N-1}}  \right) - 1 \right) \, dx  \\  & \leq \sum_{n=1}^{\infty}  \frac{1}{n!} \left( \frac{\alpha}{\mathcal{C}^{\frac{N}{N-1}}} \right)^{n} \int_{\mathcal{G}} |u(x) - (\mathcal{H}_{N,R}(u))^{\frac{1}{N}} \omega_{R}(x)|^{\frac{nN}{N-1}}  \,  dx \leq \left( \mathcal{H}_{N,R}(u) \right)^{\frac{1}{N}} \sum_{n=1}^{\infty} \frac{1}{n!}\left( \frac{nN \alpha}{N-1}  \right)^{n} .
\end{align*}
Using Stirling's approximation $n! \sim \sqrt{2 \pi n} \left( \frac{n}{e} \right)^{n}$ as $n \to \infty$, we can choose $\alpha>0$ sufficiently small such that $\sum_{n=1}^{\infty} \frac{1}{n!}\left( \frac{nN \alpha}{N-1}  \right)^{n} \leq 1$. Therefore, there exists a uniform constant $C =C (N,R, \widetilde{R}) >0$ such that 
\begin{equation*}
    \int_{\mathcal{G}}  \Phi \left( C |u(x) - (\mathcal{H}_{N,R}(u))^{\frac{1}{N}} \omega_{R}(x)| \right)  \, dx \leq \left( \mathcal{H}_{N,R}(u) \right)^{\frac{1}{N}},
\end{equation*}
where $\Phi$ is defined in \eqref{Defn: Phi}. Hence,
\begin{align}\label{Ineq delta leq 1}
    \inf_{a \geq 0} \int_{\mathcal{G}}  \Phi \left( C |u(x) - a \,\omega_{R}(x)|  \right)  \, dx    \leq   \int_{\mathcal{G}}  \Phi \left( C |u(x) - (\mathcal{H}_{N,R}(u))^{\frac{1}{N}} \omega_{R}(x)| \right)  \, dx  \leq  \left( \mathcal{H}_{N,R}(u) \right)^{\frac{1}{N}}.
\end{align}

Now, for any $x\in\mathcal{G}^{c}$, we have $v(x) \leq \left( \mathcal{H}_{N,R}(u) \right)^{\frac{1}{N}}$,  or equivalently, $u(x) \leq \left( \mathcal{H}_{N,R}(u) \right)^{\frac{1}{N}} \omega_{R}(x) $. Therefore, for any $x \in \mathcal{G}^{c}$, we have
\begin{equation*}
    |u(x) - \left( \mathcal{H}_{N,R}(u) \right)^{\frac{1}{N}} \omega_{R}(x) | \leq 2 \left( \mathcal{H}_{N,R}(u) \right)^{\frac{1}{N}} \omega_{R}(x). 
\end{equation*}
Then, for some constant $C>0$,
\begin{align*}
    \int_{\mathcal{G}^{c}} \Phi \left( C |u(x) - \left( \mathcal{H}_{N,R}(u) \right)^{\frac{1}{N}} \omega_{R}(x) |  \right) \, dx \leq \int_{\mathcal{G}^{c}} \Phi \left( 2 C \left( \mathcal{H}_{N,R}(u) \right)^{\frac{1}{N}} \omega_{R}(x)  \right) \, dx,
\end{align*}
where $\Phi$ is defined in \eqref{Defn: Phi}. Since $\Phi (st) \leq s \Phi(t)$ for all $t>0$, and $0<s \leq 1$, it follows by taking $s= \left( \mathcal{H}_{N,R}(u) \right)^{\frac{1}{N}} \leq 1$ that 
\begin{equation*}
    \int_{\mathcal{G}^{c}} \Phi \left( C |u(x) - \left( \mathcal{H}_{N,R}(u) \right)^{\frac{1}{N}} \omega_{R}(x) |  \right) \, dx \leq \left( \mathcal{H}_{N,R}(u) \right)^{\frac{1}{N}} \int_{\mathcal{G}^{c}} \Phi \left( 2 C  \omega_{R}(x)  \right) \, dx.
\end{equation*}
Choosing $C>0$ sufficiently small, we may assume that $\int_{\mathcal{G}^{c}} \Phi \left( 2 C  \omega_{R}(x)  \right) \, dx \leq 1$. Therefore, there exists a uniform constant $\mathcal{C} = \mathcal{C}(N, R, \widetilde{R})>0$ such that
\begin{align}\label{ineq5}
      \inf_{a \geq 0} \int_{\mathcal{G}^{c}}  \Phi \left( \mathcal{C} |u(x) - a \, \omega_{R}(x)|  \right) \, dx   \leq   \int_{\mathcal{G}^{c}}  \Phi \left( \mathcal{C} |u(x) - (\mathcal{H}_{N,R}(u))^{\frac{1}{N}} \omega_{R}(x)|  \right)  \, dx \leq  \left( \mathcal{H}_{N,R}(u) \right)^{\frac{1}{N}}.
\end{align}
Combining \eqref{Ineq delta leq 1} and \eqref{ineq5}, we conclude that for every $u\in C_{c}(\Omega) \cap W^{1,N}_{0}(\Omega)$ with $u\geq0$ and
$\mathcal{H}_{N,R}(u)\leq1$, there exists a constant
$\mathcal{C}=\mathcal{C}(N, R, \widetilde{R})>0$ such that
\begin{align*}
\mathcal{H}_{N,R}(u)  \geq  \left(  \inf_{a \geq 0} \int_{\Omega}  \Phi \left( \mathcal{C} |u(x) - a \, \omega_{R}(x)|  \right)  \, dx \right)^{N} .
\end{align*}

\subsection{The Case \texorpdfstring{$\mathcal{H}_{N,R}(u)> 1$}{delta > 1}:}  Now assume that $\mathcal{H}_{N,R}(u)>1$. We further assume that $\|u\|_{\operatorname{Exp} L^{\frac{N}{N-1}}(\Omega)}  \leq 1$, where $\|\cdot\|_{\operatorname{Exp}L^{\frac{N}{N-1}}(\Omega)}$ denotes the Luxemburg norm associated with the Young function $\Phi(t) = \exp \left( t^{\frac{N}{N-1}} \right)-1$, defined in \eqref{Defn: Exp Norm}. Then, by the definition of the Luxemburg
norm and $\|u\|_{\operatorname{Exp} L^{\frac{N}{N-1}}(\Omega)}  \leq 1$,
\begin{align*} 
\left( \mathcal{H}_{N,R}(u) \right)^{\frac{1}{N}} > 1  \geq \int_{\Omega}  \Phi \left( \frac{|u(x)|}{\|u\|_{\operatorname{Exp} L^{\frac{N}{N-1}}(\Omega)}} \right) dx  & \geq  \int_{\Omega}  \Phi \left( |u(x)| \right) dx \\ & \geq \inf_{a \geq 0}  \int_{\Omega}  \Phi \left(  |u(x) - a \, \omega_{R}(x)|  \right)  \, dx  .
\end{align*}
Therefore, in the case $\mathcal{H}_{N,R}(u)>1$, under the normalization
$\|u\|_{\operatorname{Exp}L^{\frac{N}{N-1}}(\Omega)} \leq 1$, we obtain
\begin{align*}
 \mathcal{H}_{N,R}(u)  \geq  \left(  \inf_{a \geq 0} \int_{\Omega}  \Phi \left(  |u(x) - a \, \omega_{R}(x)| \right)  \, dx \right)^{N} .
\end{align*}
Hence, combining the cases $\mathcal{H}_{N,R}(u)\leq1$ and $\mathcal{H}_{N,R}(u)>1$, we conclude that for every $u\in C_{c}(\Omega) \cap W^{1,N}_{0}(\Omega)$ with $u\geq0$ and
$\|u\|_{\operatorname{Exp}L^{\frac{N}{N-1}}(\Omega)} \leq 1$,
\begin{align}\label{ineq8}
 \mathcal{H}_{N,R}(u)  \geq  \left(  \inf_{a \geq 0} \int_{\Omega}  \Phi \left( C |u(x) - a \, \omega_{R}(x)|  \right)  \, dx \right)^{N} .
\end{align}
Now let $u\in C_{c}(\Omega) \cap W^{1,N}_{0}(\Omega)$ satisfy $\|u\|_{\operatorname{Exp}L^{\frac{N}{N-1}}(\Omega)} \leq 1$, and define $u_{+}(x) = \max \{ u(x), 0 \}$ and $u_{-} = \max \{ -u(x), 0 \}$. Then $u_{+},u_{-}\geq0$ and
\begin{equation*}
    \max \left\{ \|u_{+}\|_{\operatorname{Exp}L^{\frac{N}{N-1}}(\Omega)} \, , \, \|u_{-}\|_{\operatorname{Exp}L^{\frac{N}{N-1}}(\Omega)} \right\} \leq \|u\|_{\operatorname{Exp}L^{\frac{N}{N-1}}(\Omega)} \leq 1.
\end{equation*}
Hence, using the convexity of
$\Phi$, defined in \eqref{Defn: Phi}, the decomposition $u=u_{+}-u_{-}$, and the inequality \eqref{ineq8}, we obtain
\begin{align*}
 & \left(  \inf_{a \in \mathbb{R}} \int_{\Omega}   \Phi \left( \frac{C}{2}  |u(x) - a \, \omega_{R}(x)| \right) \, dx \right)^{N} \\   & \hspace{3cm} = \left(  \inf_{b,c \geq 0} \int_{\Omega}  \Phi \left( \frac{C}{2}  |u_{+}(x) - u_{-}(x) - (b-c) \, \omega_{R}(x)| \right) \, dx \right)^{N} \\ & \hspace{3cm} \leq \left( \frac{1}{2}  \sum_{\pm} \inf_{a \geq 0} \int_{\Omega}   \Phi \left( C  |u_{\pm}(x) - a \, \omega_{R}(x)|  \right)  \, dx \right)^{N} \\ & \hspace{3cm} \leq \frac{1}{2} \left(\mathcal{H}_{N,R}(u_{+}) + \mathcal{H}_{N,R}(u_{-}) \right) \leq  \mathcal{H}_{N,R}(u).
\end{align*}
Finally, to remove the assumption
$\|u\|_{\operatorname{Exp}L^{\frac{N}{N-1}}(\Omega)}\leq1$,
apply the above inequality to
\begin{equation*}
    \widetilde{u} = \frac{u}{\|u\|_{\operatorname{Exp}L^{\frac{N}{N-1}}(\Omega)}}.
\end{equation*}
Using the definition of $\Phi$, we conclude that for any $u \in C_{c}(\Omega) \cap W^{1,N}_{0}(\Omega)$, there exists a constant
$C=C(N,R,\widetilde{R})>0$ such that
\begin{align*}
 \mathcal{H}_{N,R}(u) &  \geq \|u\|^{N}_{\operatorname{Exp}L^{\frac{N}{N-1}}(\Omega)}  \left(  \inf_{a \in \mathbb{R}} \bigintsss_{\Omega}  \Phi \left( C \frac{|u(x) - a \, \omega_{R}(x)|}{\|u\|_{\operatorname{Exp}L^{\frac{N}{N-1}}(\Omega)}}   \right)  \, dx \right)^{N}  \\ & = \|u\|^{N}_{\operatorname{Exp}L^{\frac{N}{N-1}}(\Omega)} \, \bigl(\mathcal{D}_{N,R,C}(u)\bigr)^{N}  .
\end{align*}
Let $u \in W^{1,N}_{0}(\Omega)$. By density, there exists a sequence
$\{u_{n}\} \subset C^{\infty}_{c}(\Omega)$ such that
\begin{equation*}
\|u_{n}-u\|_{W^{1,N}(\Omega)} \to 0,
\qquad \text{as} \quad n\to\infty.
\end{equation*}
Using the continuous embedding $W^{1,N}_{0}(\Omega)
\subset
\operatorname{Exp}L^{\frac{N}{N-1}}(\Omega)$, we have
\begin{equation*}
  \|u_{n}-u\|_{\operatorname{Exp}L^{\frac{N}{N-1}}(\Omega)}
\to 0 
\quad \text{as} \quad n\to\infty.
\end{equation*}
Moreover, up to a subsequence, $u_n\to u$ almost everywhere in $\Omega$. Therefore, for every fixed $a\in\mathbb{R}$, Fatou's lemma yields
\begin{align*}
& \bigintsss_{\Omega} \Phi \left(
\frac{ C|u(x)-a \, \omega_R(x)|}{ \|u\|_{\operatorname{Exp}L^{\frac{N}{N-1}}(\Omega)}}\right) \, dx  \leq \liminf_{n\to\infty} \bigintsss_{\Omega}  \Phi \left(\frac{ C|u_n(x)-a \, \omega_R(x)|}{ \|u_n\|_{\operatorname{Exp}L^{\frac{N}{N-1}}(\Omega)}}\right) \, dx.
\end{align*}
Taking the infimum over $a\in\mathbb{R}$, we obtain
\begin{equation*}
\mathcal{D}_{N,R,C}(u)
\leq
\liminf_{n\to\infty}\mathcal{D}_{N,R,C}(u_n).
\end{equation*}
Therefore, passing to the limit in the inequality established for
$u_{n}\in C^{\infty}_{c}(\Omega)$, we obtain
\begin{equation*}
\mathcal{H}_{N,R}(u)
\geq
\|u\|^{N}_{\operatorname{Exp}L^{\frac{N}{N-1}}(\Omega)}
\bigl(\mathcal{D}_{N,R,C}(u)\bigr)^{N}.
\end{equation*}
This completes the proof of the theorem.
\end{proof}

\begin{proof}[\textbf{Proof of Theorem \ref{Theorem 2}}]
Let $\Omega$ be a bounded open set containing the origin, and let
$\widetilde{R}=\sup_{x\in\Omega}|x|$ with $R\geq\widetilde{R}$.
Let $u\in C_{c}(\Omega) \cap W^{1,N}_{0}(\Omega)$ with $u\geq0$. By extending $u$ by zero outside $\Omega$, we may regard $u$ as an element of $C_c(B_{\widetilde{R}}(0))\cap W^{1,N}_{0}(\Omega)$. Assume that
\begin{equation*}
  \| u\|_{\mathcal{L}^{N, R}(\Omega)}
= \left( \frac{N}{N-1} \right),
\end{equation*}
where $\| u\|_{\mathcal{L}^{N, R}(\Omega)}$ is defined in \eqref{Hardy Littlewood inequality}. Using the definition of the Schwarz symmetrization $u^{\#}(x)=u^{*}\bigl(|B_{1}(0)||x|^{N}\bigr)$ and passing to polar coordinates followed by the change of variables $s=|B_{1}(0)|r^{N}$, we obtain
\begin{align*}
\int_{B_{\widetilde R}(0)}
\frac{|u^{\#}(x)|^{N}}
{|x|^{N}\left(1+\log\frac{R}{|x|}\right)^{N}}\,dx
=|B_{1}(0)|N^{N}
\int_{0}^{|B_{\widetilde R}(0)|}
\frac{|u^{*}(s)|^{N}}
{\left(N+\log\frac{|B_{1}(0)|R^{N}}{s}\right)^{N}}
\,\frac{ds}{s}.
\end{align*}
Consequently, by the definition of the Lorentz--Zygmund norm in
\eqref{Lorentz Zygmund Norm},
\begin{equation*}
\| u\|_{\mathcal{L}^{N, R}(\Omega)}
=
|B_{1}(0)|^{\frac{1}{N}}N
\|u\|_{L^{\infty,N}(\log L)^{-1}(B_{\widetilde{R}}(0)),R}.
\end{equation*}
Using the continuous embedding $L^{\infty,N}(\log L)^{-1}(\Omega)
\subsetneq 
\operatorname{Exp}L^{\frac{N}{N-1}}(\Omega)$ (see \cite[Equation $1.16$]{CianchiFerone}), together with the above estimate, and $\operatorname{supp} u \subset \Omega$, there exists a constant $C=C(N,R, \widetilde{R})>0$ such that
\begin{equation}\label{ineq11}
    \| u\|_{\operatorname{Exp}L^{\frac{N}{N-1}}(\Omega)} \leq C \| u\|_{\mathcal{L}^{N, R}(\Omega)}.
\end{equation}

Define the Hardy deficit by
\begin{equation}\label{Defn: Hardy deficit 2}
    \widetilde{\mathcal{H}}_{N, R}(u)
    :=
    \int_{\Omega} |\nabla u(x)|^{N}\,dx
    -
    \left(\frac{N-1}{N}\right)^{N}
    \bigintsss_{\Omega}
    \frac{|u(x)|^{N}}
    {|x|^{N}\left( 1+ \log\frac{R}{|x|} \right)^{N}}
    \,dx.
\end{equation}
First, assume that $\widetilde{\mathcal{H}}_{N,R}(u)\leq1$. 
Replacing $R$ by $Re>\widetilde{R}$ in the Hardy deficit 
$\mathcal{H}_{N,R}(u)$, where $R\geq\widetilde{R}$, and using $\log \frac{Re}{|x|} = 1+ \log \frac{R}{|x|}$, we obtain
\begin{equation*}
 \mathcal{H}_{N, Re}(u)  =  \widetilde{\mathcal{H}}_{N, R}(u) \leq 1 .
\end{equation*}
Therefore, by the proof of Theorem~\ref{Theorem: Quantitative stability critical Hardy}, under the assumption
$\mathcal{H}_{N,Re}(u)\leq1$, for any $u\in C_{c}(\Omega) \cap W^{1,N}_{0}(\Omega)$ with $u\geq0$ and
$\widetilde{\mathcal{H}}_{N,R}(u)\leq1$, we have
\begin{align*}
\int_{\Omega} |\nabla u(x)|^{N}\,dx - & \left(\frac{N-1}{N}\right)^{N} \bigintsss_{\Omega}
\frac{|u(x)|^{N}} {|x|^{N}\left(1+\log\frac{R}{|x|}\right)^{N}} \,dx  \\ & \geq \left( \inf_{a\geq0}
\int_{\Omega} \left(
\exp\left( C|u(x)-a\, \widetilde{\omega}_R(x)|^{\frac{N}{N-1}} \right)-1
\right) \,dx \right)^N,
\end{align*}
where $C=C(N,R, \widetilde{R})>0$ and $\widetilde{\omega}_{R}$ is defined in \eqref{Defn: Omege widetilde}.

Now, assume that $\widetilde{\mathcal{H}}_{N,R}(u)>1$. Using the normalization $ \| u\|_{\mathcal{L}^{N, R}(\Omega)} = \left( \frac{N}{N-1} \right)$, the inequality \eqref{ineq11}, and the function $\Phi$, defined in \eqref{Defn: Phi}, we obtain
\begin{align*}
   \left(  \widetilde{\mathcal{H}}_{N, R}(u) \right)^{\frac{1}{N}} > 1  & \geq \int_{\Omega} \Phi \left(\ \frac{|u(x)|}{\| u\|_{\operatorname{Exp}L^{\frac{N}{N-1}}(\Omega)}} \right) \, dx \geq \int_{\Omega} \Phi \left( C |u(x)| \right) \, dx \\ &  \geq  \inf_{a \geq 0} \int_{\Omega}  \Phi \left( C  |u(x) - a \, \widetilde{\omega}_{R}(x)|  \right) \, dx .
\end{align*}
Combining the cases $\widetilde{\mathcal{H}}_{N,R}(u)\leq1$ and
$\widetilde{\mathcal{H}}_{N,R}(u)>1$, we conclude that, for any
$u\in C_{c}(\Omega) \cap W^{1,N}_{0}(\Omega)$ with $u\geq0$ satisfying $ \| u\|_{\mathcal{L}^{N, R}(\Omega)} = \left( \frac{N}{N-1} \right)$, we have
\begin{equation*}
  \widetilde{\mathcal{H}}_{N, R}(u) \geq   \left(  \inf_{a \geq 0} \int_{\Omega}  \Phi \left( C |u(x) - a \, \widetilde{\omega}_{R}(x)| \right)  \, dx \right)^{N},
\end{equation*}
where $C=C(N,R, \widetilde{R})>0$. 

To remove the assumption
$ \| u\|_{\mathcal{L}^{N, R}(\Omega)} = \left( \frac{N}{N-1} \right)$,
apply the above inequality to
\begin{equation*}
    \widetilde{u} = \frac{u}{\| u\|_{\mathcal{L}^{N, R}(\Omega)} } \left( \frac{N}{N-1} \right) ,
\end{equation*}
we obtain
\begin{equation}\label{ineq10}
  \widetilde{\mathcal{H}}_{N, R}(u) \geq \left( \frac{N-1}{N} \right)^{N}  \| u\|^{N}_{\mathcal{L}^{N, R}(\Omega)}  \left(  \inf_{a \geq 0} \int_{\Omega}  \Phi \left( C \frac{|u(x) - a \, \widetilde{\omega}_{R}(x)|}{\| u\|_{\mathcal{L}^{N, R}(\Omega)}}  \right) \, dx \right)^{N}  .
\end{equation}
Now let $u\in C_{c}(\Omega) \cap W^{1,N}_{0}(\Omega)$, and define $u_{+}(x) = \max \{ u(x), 0 \}$ and $u_{-} = \max \{ -u(x), 0 \}$. Then $u_{+},u_{-}\geq0$.  

Hence, using the convexity of
$\Phi$, defined in \eqref{Defn: Phi}, and the decomposition $u=u_{+}-u_{-}$, we obtain
\begin{align*}
 & \left( \frac{N-1}{N} \right)^{N}  \| u\|^{N}_{\mathcal{L}^{N, R}(\Omega)} \left(  \inf_{a \in \mathbb{R}} \int_{\Omega}   \Phi \left( \frac{C}{2}  \frac{|u(x) - a \, \widetilde{\omega}_{R}(x)|}{ \| u\|_{\mathcal{L}^{N, R}(\Omega)}}  \right) \, dx \right)^{N} \\   & \hspace{1cm} = \left( \frac{N-1}{N} \right)^{N}  \| u\|^{N}_{\mathcal{L}^{N, R}(\Omega)} \left(  \inf_{b,c \geq 0} \int_{\Omega}  \Phi \left( \frac{C}{2}  \frac{|u_{+}(x) - u_{-}(x) - (b-c) \, \widetilde{\omega}_{R}(x)|}{ \| u\|_{\mathcal{L}^{N, R}(\Omega)}}  \right) \, dx \right)^{N} \\ & \hspace{1cm} \leq \frac{1}{2} \sum_{\pm} \left( \frac{N-1}{N} \right)^{N}   \left(  \inf_{a \geq 0} \frac{\bigintsss_{\Omega}   \Phi \left( C \frac{|u_{\pm}(x) - a \, \widetilde{\omega}_{R}(x)|}{\| u\|_{\mathcal{L}^{N, R}(\Omega)}}   \right)  \, dx}{\| u\|^{-1}_{\mathcal{L}^{N, R}(\Omega)}}  \right)^{N}.
\end{align*}
Also,  $t \mapsto \frac{\Phi(st)}{t}$ is increasing for any fixed $s \geq 0$. Therefore, using this fact, $\|u_{\pm}\|_{\mathcal{L}^{N, R}(\Omega)} \leq \|u\|_{\mathcal{L}^{N, R}(\Omega)}$, and the inequality \eqref{ineq10} for $u_{+}$ and $u_{-}$, the above inequality reduces to
\begin{align*}
   & \left( \frac{N-1}{N} \right)^{N} \| u\|^{N}_{\mathcal{L}^{N, R}(\Omega)} \left(  \inf_{a \in \mathbb{R}} \int_{\Omega}   \Phi \left( \frac{C}{2}  \frac{|u(x) - a \, \widetilde{\omega}_{R}(x)|}{ \| u\|_{\mathcal{L}^{N, R}(\Omega)}}  \right) \, dx \right)^{N} \\ & \hspace{2cm} \leq  \frac{1}{2} \sum_{\pm} \left( \frac{N-1}{N} \right)^{N}   \left(  \inf_{a \geq 0} \frac{\bigintsss_{\Omega}   \Phi \left( C \frac{|u_{\pm}(x) - a \, \widetilde{\omega}_{R}(x)|}{\| u_{\pm}\|_{\mathcal{L}^{N, R}(\Omega)}}   \right)  \, dx}{\| u_{\pm}\|^{-1}_{\mathcal{L}^{N, R}(\Omega)}}  \right)^{N} \\ & \hspace{2cm}\leq \frac{1}{2} \left( \widetilde{\mathcal{H}}_{N,R}(u_{+}) + \widetilde{\mathcal{H}}_{N,R}(u_{-}) \right) \leq \widetilde{\mathcal{H}}_{N,R}(u).
\end{align*}
From the inequality \eqref{Hardy Littlewood inequality}, the above estimate, and the definition of $\widetilde{\mathcal{H}}_{N,R}(u)$, defined in \eqref{Defn: Hardy deficit 2}, we obtain
\begin{align*}
 \int_{\Omega} |\nabla u(x)|^{N}\,dx
-
\left(\frac{N-1}{N}\right)^{N} &
\bigintsss_{\Omega}
\frac{|u(x)|^{N}}
{|x|^{N} \left( 1+ \log \frac{R}{|x|} \right)^{N}}
\,dx \\ &
\geq
\left(\frac{N-1}{N}\right)^{N} \left(
\bigintsss_{\Omega}
\frac{|u(x)|^{N}}
{|x|^{N} \left( 1+ \log \frac{R}{|x|} \right)^{N}}
\,dx \right)  \bigl(\widetilde{\mathcal{D}}_{N,R,C}(u)\bigr)^{N}.
\end{align*}
Similarly, following the same approach as in the proof of Theorem \ref{Theorem: Quantitative stability critical Hardy}, the above inequality holds for every $u \in W^{1,N}_{0}(\Omega)$. This completes the proof of the theorem. 
\end{proof}

\bigskip

\textbf{Acknowledgement:} The author gratefully acknowledges the financial support of the Anusandhan National Research Foundation (ANRF) through the National Postdoctoral Fellowship (PDF/2025/004611). The author also thanks the Theoretical Statistics and Mathematics Unit, Indian Statistical Institute, Delhi Centre, India, for providing a supportive and stimulating research environment. The author is grateful to Prof. Debdip Ganguly for their helpful discussions on this topic.


\end{document}